\documentclass[12pt,oneside,reqno]{amsart}
\usepackage{amsfonts,amsmath,amssymb,amscd}
\usepackage{enumerate}
\usepackage{enumitem}

\newtheorem{theorem}{Theorem}[section]
\newtheorem{lemma}[theorem]{Lemma}

\newcommand{\Z}{\mbox{$\mathbb Z$}}

\newcommand\numberthis{\addtocounter{equation}{1}\tag{\theequation}}

\begin{document}
	\title[]{Sums of $S$-units in $X$-coordinates of Pell equations}

	\author[Nair]{Parvathi S Nair}
	\address{Parvathi S Nair, Department of Mathematics, National Institute of Technology Calicut, 
		Kozhikode-673 601, India.}
	\email{parvathisnair60@gmail.com; parvathi\_p220245ma@nitc.ac.in}

	\author[Rout]{S. S. Rout}
	\address{Sudhansu Sekhar Rout, Department of Mathematics, National Institute of Technology Calicut, 
		Kozhikode-673 601, India.}
	\email{sudhansu@nitc.ac.in; lbs.sudhansu@gmail.com}

	\dedicatory{}
	\begin{abstract}
		Let $S$ be a fixed set of primes and let $(X_{l})_{l\geq 1}$ be the $X$-coordinates of the positive integer solutions $(X, Y)$ of the Pell equation $X^2-dY^2 = 1$ corresponding to a non-square integer $d>1$. We show that  there are only a finite number of non-square integers $d>1$ such that there are at least two different elements of the sequence $(X_{l})_{l\geq 1}$ that can be represented as a sum of $S$-units with a fixed number of terms. Furthermore, we solve explicitly a particular case in which two of the $X$-coordinates are product of power of two and power of three.
		
	\end{abstract}
	\maketitle
	\pagenumbering{arabic}
	\pagestyle{headings}
	
	\section{Introduction}
	
	The problem of finding specific terms of a linear recurrence sequence of some particular form has a very rich history.  Peth\H{o} \cite{Petho1982} and Shorey-Stewart \cite{Shorey1983} independently studied the perfect powers in a non-degenerate binary recurrence sequence. In particular, they have considered the Diophantine equation
	\begin{equation}\label{eq0a}
		U_{n}  = x^{z}
	\end{equation}
	in integers $n,x,z$ with $z\geq 2$, where $( {U_{n})_{n \geq 0} }$ is a non-degenerate binary recurrence sequence and proved under certain natural assumptions that \eqref{eq0a} contain only finitely many perfect powers. For example, Fibonacci and Lucas numbers, respectively, of  the form $x^{z}$, with $z > 1$ has been recently proved by Bugeaud et al. \cite{Bugeaud2006}. Peth\H{o} \cite{Petho1991} proved that there are no non-trivial perfect powers in Pell sequence (see also \cite{Cohn1996}). Peth\H{o} and Tichy \cite{Petho1993} proved that there are only finitely  many  Fibonacci  numbers  of  the  form $p^{a} + p^b + p^c$, with prime $p$ and integers $a>b>c\geq 0$. Marques and Togb\'{e} \cite{Marques2013} found all Fibonacci and Lucas numbers of the form $2^a + 3^b + 5 ^c$ under the condition that $c\geq \max\{a, b\}\geq 0$.

	Recently,  Bert\'{o}k et al., \cite{bhpr2} under some mild assumptions gave finiteness result for the solutions of the Diophantine equation 
	\begin{equation}\label{eq0b}
		U_{n}  = b_1 p_1^{z_1} + \cdots+ b_s p_s^{z_s}
	\end{equation}
	in non-negative integers $z_1, \ldots, z_s$ and $n$, where $(U_{n})_{n\geq 0}$ is a binary non-degenerate recurrence sequence of positive discriminant, $p_{1}, \ldots, p_{s}$ are given primes and $b_1, \ldots, b_s$ are fixed integers. 
	
	Diophantine equations combining both $S$-units and recurrence sequences has been well studied by many authors (see \cite{bhpr1, bhpr2, EGL2020, hs}). Suppose $S$ is the set of distinct primes $p_1,\ldots,p_{l}$. Then a rational integer $z$ is an $S$-unit if $z$ can be written as
	\begin{equation}\label{eq0}
		z = \pm p_{1}^{e_1}\cdots p_{l}^{e_{l}},
	\end{equation}
	where $e_1, \ldots, e_{l}$ are non-negative integers and we denote the set of $S$-units by $\mathcal{U}_S$. Let $(U_n)_{n\geq 0}$ be a recurrence sequence of order $k$ with $k\geq 2$.  B\'erczes et al. \cite{bhpr1} considered the equation 
	\begin{equation}\label{eq0sunit}
		U_n = z_1+\cdots+z_r,
	\end{equation}
	with some arbitrary but fixed $r\geq 1$, in unknown $n\geq 0$ and $z_1, \ldots, z_r\in \mathcal{U}_S$ and establish finiteness results for the solutions of \eqref{eq0sunit}.
	
	Let $d, t$ be nonzero integers with $d > 1$ square-free. Consider the Pell equation
	\begin{equation}\label{gpelleqn}
		x^2 - dy^2 = 1
	\end{equation}
	in integers $x, y$. Recently, several mathematicians have investigated the following type of problem related to solution sets (i.e., $X$ and $Y$-coordinates) of Pell equation (\ref{gpelleqn}).  Assume that $\mathbb{U}:= (U_n)_{n\geq 0}$ is some interesting sequence of positive integers and $\{(X_m, Y_m)\}_{m\geq 1}$ are sequence of solutions of Pell equation \eqref{gpelleqn}. What can one say about the number of solutions of the containment $X_m\in \mathbb{U}$ for a generic $d$? What about the number of solutions of the containment $Y_m\in \mathbb{U}$? For most of the binary recurrent sequences(Fibonacci numbers \cite{ birlucatogbe2018, lucatogbe2018}, tribonacci numbers \cite{luca 2017}, rep-digits in some given integer base $b\geq 2$ \cite{dlt2016, fl2018, glz2020}), the equation $X_m\in \mathbb{U}$ has at most one positive integer solution $m$ for any given $d$ except for finitely many values of $d$.  Erazo et al., \cite{EGL2020} showed under certain conditions that there are only a finite number non-square integers $d>1$ such that there are at least two different elements of the sequence $(X_m)_{m\geq 1}$ that can be represented as a linear combination of prime powers with fixed primes, i.e., 
	\[X_m  = c_1p_1^{n_1}+\cdots + c_sp_s^{n_s}.\]
	
	In this paper, we extend this result to sums of $S$-units. In particular, we will prove that there are only a finite number non-square integers $d>1$ such that there are at least two different elements of the sequence $(X_m)_{m\geq 1}$ that can be represented as a sum of $S$-units. Firstly, we prove the general case when elements of the sequence $(X_m)_{m\geq 1}$ equal to  sum of $S$-units, that is, 
	\begin{equation}\label{eq3a}
		X_{l}  = z_1+\cdots+ z_r,
	\end{equation} 
	where $z_1, \ldots, z_r\in \mathcal{U}_S$   and $l$ are positive integers. Then in next result, we solve a particular case of \eqref{eq3a} with $S=\{2, 3\}$. 
	\section{Notation and Main Results} 
	Let $d>1$ be an integer which is not a square. The Pell equation
	\begin{equation}\label{eqpell}
		X^2-dY^2 = 1, 
	\end{equation}
	where $X, Y \in \Z_{>0}$ has infinitely many positive integer solutions $(X, Y)$ and  have the form 
	\[X+Y \sqrt{d} = X_l+Y_l\sqrt{d} = (X_1+Y_1\sqrt{d})^l\] for some $k\in \Z_{>0}$ and $(X_1, Y_1)$ is the smallest positive integer solution of \eqref{eqpell}. The sequence $(X_l)_{l\geq 1}$ is a binary recurrence sequence satisfying the recurrence relation 
	\begin{equation}
		X_{l} = 2X_1 X_{l-1}-X_{l-2}
	\end{equation}
	for all $l\geq 2$. Setting
	\begin{equation}
		\gamma:= X_1+Y_1\sqrt{d} \quad \mbox{and}\quad \eta:= X_1 -Y_1 \sqrt{d} = \gamma^{-1},
	\end{equation}
	so $\gamma \cdot \eta = X_1^2-dY_1^2= 1$. The Binet formula for $X_{l}$ and $Y_l$ are 
	\begin{equation}\label{eq7a}
		X_{l} = \frac{\gamma^l+ \eta^l}{2} \quad \mbox{and}\quad Y_{l} = \frac{\gamma^l - \eta^l}{2\sqrt{d}}
	\end{equation} holds for all non-negative integers $l$. Let $s\in \Z_{>0}$  be fixed. We are interested to determine for which positive integers $d>1$, the sequence $(X_{l})_{l\geq 1}$ of $X$-coordinates of \eqref{eqpell} has at least two different terms that can be represented as in \eqref{eq3a}. By denoting $z_i= p_{1}^{n_{i1}}\cdots p_{s}^{n_{is}}$ for all $1 \leq i \leq r$, we write \eqref{eq3a} as
	\begin{equation}\label{eq3d}
		X_{l}= p_1^{n_{11}}p_2^{n_{12}}\cdots p_s^{n_{1s}}+\cdots+ p_1^{n_{r1}}p_2^{n_{r2}}\cdots  p_s^{n_{rs}}.
	\end{equation} 
	such that,
	\begin{equation}\label{eq3c}
		n_{rs}=\max_{\substack{1\leq i \leq r\\1\leq j \leq s}}\{n_{ij}\}.
	\end{equation}
	Now we are ready to state our main theorem. 
	\begin{theorem}\label{th1}
		Let $s$ be a fixed positive integer. Let $p_1 \leq \cdots \leq p_s$ be fixed primes with $p_s$ odd. Let $X_{l}$ be the $X$-coordinate of the Pell equation \eqref{eqpell} with $d>1$ non-square. Let $\epsilon>0$ be arbitrary. Consider \eqref{eq3a} in $l\geq 1$ and $z_1, \ldots, z_r$ satisfying $|z_i|^{1+\epsilon}<|z_r|, ( i=1,\ldots ,r-1)$ and \eqref{eq3c}. 
		Let $ T_{d}$ be the set of solutions $ (l, z_1, \ldots, z_r) $ of \eqref{eq3a}.
		Then there exists effectively computable constants $c_8(s)$ and $c_{9}(s)$ depend only on the parameters $s, p_s, r \text{ and }\epsilon$ such that
		\begin{enumerate}[label=\textnormal{(\roman*)}]
			\item if $d <p_s^{2c_8}r^2$, then $\# T_{d} \leq c_9(s)$,
			\item if $d \geq p_s^{2c_8}r^2$, then $\# T_{d} \leq  1$.
		\end{enumerate}
	\end{theorem}
	Note that Theorem \ref{th1} extends the corresponding statement from \cite{EGL2020} to this more general situation. Our next theorem illustrates Theorem \ref{th1}, in which we explicitly solve \eqref{eq3d} for the case $r=1, s=2$ and $S=\{2,3\}$.
	\begin{theorem}\label{th2}
		Let $X_{l}$ be the $X$-coordinate of the Pell equation \eqref{eqpell} with $d>1$ non-square. Then there exist no $d$ for which the equation 
		\begin{equation}\label{eq3b}
			X_{l}=2^{n_1}3^{n_2}, \qquad {n_1\leq n_2}
		\end{equation}
		has at least two positive integer solutions. 
	\end{theorem}
	
	\section{Auxiliary results}
	In this section, we will prove some preliminary results to prove main theorems.
	\begin{lemma}\label{lem1}
		Let $\gamma >0$ be the fundamental solution of \eqref{eqpell} for $d>1$ non-square. Then 
		\begin{equation}\label{eq6}
			\left(\frac{1}{1+\sqrt{2}}\right)\gamma^{l} \leq X_{l} \leq (2-\sqrt{2})\gamma^{l}, \quad \mbox{for all $l \geq 1$}.
		\end{equation}
	\end{lemma}
	\begin{proof}
		See Lemma 1 in \cite{EGL2020}. 
	\end{proof}
	\par
	Let $ \alpha $ be an algebraic number of degree $d$. Then the \textit{ logarithmic height}  of the algebraic number $\alpha$ is given  by
	$$ 
	h(\alpha) = \frac{1}{d} \left( \log |a| + \sum_{i=1}^d \log \max \{ 1, |\alpha^{(i)}| \} \right),
	$$
	where $a$ is the leading coefficient of the minimal polynomial of $\alpha$ over $\mathbb{Z}$ and the $\alpha^{(i)}$'s are the conjugates of $\alpha$ in $\mathbb{C}$. In particular, if $z= p/q$ is a rational number with $\gcd(p, q) =1$, and $q>0$ then $h(z) = \log \max \{|p|, q\}$. 
	
	To prove Theorem \ref{th1}, we use a lower bound of linear forms in logarithms to bound the index $l, z_1, \ldots, z_s$ appearing in \eqref{eq3a}.  In particular, we need the following result due to Matveev \cite{Matveev2000}.
	\begin{lemma}[Matveev's Theorem]\label{lem12}
		Let $\eta_1,\ldots,\eta_t$ be real algebraic numbers and let $d_{1},\ldots, d_{t}$ be rational integers. Let $d_{\mathbb{L}}$ be the degree of the number field $\mathbb{L}=\mathbb{Q}(\eta_1,\ldots,\eta_t)$ over $\mathbb{Q}$. Let $A_{j}$ be real numbers satisfying 
		\[ A_j \geq \max \left\{ d_{\mathbb{L}}h(\eta_j) , |\log \eta_j|, 0.16  \right\}, \quad j= 1, \ldots,t.\]
		Assume that $B\geq \max\{|d_1|, \ldots, |d_{t}|\}$ and $\Lambda:=\eta_{1}^{d_1}\cdots\eta_{t}^{d_t} - 1$. If $\Lambda \neq 0$, then
		\[|\Lambda| \geq \exp \left( -1.4\cdot 30^{t+3}\cdot t^{4.5}\cdot d_{\mathbb{L}}^{2}(1 + \log d_{\mathbb{L}})(1 + \log B)A_{1}\cdots A_{t}\right).\]
	\end{lemma}
	
	
	\begin{lemma}\label{lemma3.3}
		Let $B$ be a non-negative integer such that $\delta B \leq \alpha \log B+ \beta$. If $\alpha \geq e \delta$, then 
		$$B \leq (2/{\delta})(\alpha \log(\alpha/{\delta})+\beta).$$
	\end{lemma}
	\begin{proof}
		See \cite[Lemma 12.2.4]{cohen2007}
	\end{proof}
	We need some more results from the theory of continued fractions. The following result will be  useful for the treatment of linear forms in logarithms.
	\begin{lemma}\label{lemma3.4}
		Let $\tau$ be an irrational number, $M$ be a positive integer  and $p_0/q_0, p_1/q_1,\ldots,$ be all the convergents of the continued fraction $[a_0,a_1,\ldots]$ of $\tau$. Let $N$ be such that $q_N>M$. Then putting $a(M):=\max\{a_t:t=0,1,\ldots, N\},$ the inequality 
		$$|m\tau-n|>\dfrac{1}{(a(M)+2)m}$$ holds for all pairs $(n, m)$ of integers with $0<m<M$.
	\end{lemma}
	After getting the upper bound of $n_2$, which is generally too large, the next step is to reduce it. For this reduction purpose, we use a variant of the Baker–Davenport result \cite{Baker1969}. Here, for a real number $x$, let $||x||:=\min\{|x-n|: n\in \mathbb{Z}\}$ denote the distance from $x$ to the nearest integer.
	
	\begin{lemma}[See \cite{Dujella1998}]\label{lemma3.5}
		Let $\tau$ be an irrational number, $M$ be a positive integer  and $p/q$ be a convergent of the continued fraction of $\tau$ such that $q>6M$. Let $A, B, \mu $ be some real numbers with $A>0$ and $B>1$. Put $\epsilon_1:=||\mu q||-M||\tau q||$, then there is no solution to the inequality 
		$$|m\tau-n+\mu|<AB^{-k},$$ in positive integers $m,n$ and $k$ with
		$$m\leq M \qquad \text{ and } \qquad k\geq \dfrac{\log(Aq/{\epsilon_1})}{\log B}.$$
	\end{lemma}
	
	\section{Proof of Theorem \ref{th1}}
	Suppose $1 \leq l_1 < l_2$, such that
	\begin{equation}\label{eq 14}
		X_{l_1}= p_1^{a_{11}}p_2^{a_{12}}\cdots p_s^{a_{1s}}+\cdots+ p_1^{a_{r1}}p_2^{a_{r2}}\cdots  p_s^{a_{rs}}$$ and 
		$$ X_{l_2}= p_1^{b_{11}}p_2^{b_{12}}\cdots p_s^{b_{1s}}+\cdots+ p_1^{b_{r1}}p_2^{b_{r2}}\cdots  p_s^{b_{rs}}.
	\end{equation}
	Denote 
	\begin{align*}
		(l, n_{11},\ldots, n_{1s},n_{21},&\ldots, n_{2s}, \ldots, n_{r1},\ldots, n_{rs})\\
		&  =(l_1,a_{11},\ldots, a_{1s}, a_{21},\ldots a_{2s},\ldots, a_{r1},\ldots, a_{rs}) 
	\end{align*}
	or
	\begin{align*}
		(l, n_{11},\ldots, n_{1s},n_{21},&\ldots, n_{2s}, \ldots, n_{r1},\ldots, n_{rs})\\
		&  =(l_2, b_{11},\ldots, b_{1s}, b_{21},\ldots b_{2s},\ldots, b_{r1},\ldots, b_{rs}).
	\end{align*}
	By Lemma \ref{lem1},
	\begin{equation}\label{eq31}
		\dfrac{\gamma ^l}{2.5}<X_l<\gamma^l.
	\end{equation}
	From \eqref{eq3d} and \eqref{eq3c}, it follows that,
	\begin{equation}\label{eq032}
		p_s^{n_{rs}}< X_l< r  p_s^{s n_{rs}}.
	\end{equation}
	As $(X_l)_{l\geq 1}$ is an increasing sequence,
	\[p_s^{a_{rs}}\leq X_{l_1} \leq X_{l_2}\leq r p_s^{s b_{rs}}\] and this implies
	\begin{equation}\label{eq107}
		a_{rs}\leq s b_{rs}+\dfrac{\log r}{\log p_s}.
	\end{equation}
	From (\ref{eq31}) and (\ref{eq032}) we get
	$$p_s^{n_{rs}}< \gamma^{l}< 2.5X_l< 2.5 r p_s^{s n_{rs}},$$ and this further implies,
	\begin{equation}\label{eq108}
		n_{rs}\log p_s<l \log \gamma < s n_{rs}\log p_s+\log 2.5 r.
	\end{equation} Thus,
	\begin{equation}\label{eq109}
		\dfrac{n_{rs}}{\log \gamma}<\dfrac{l}{\log p_s}<l.
	\end{equation}
	Since for $ \epsilon>0, \;|z_i|^{1+\epsilon}<|z_r|, \text{ for }$ $i=1,\ldots, r-1$, then from \eqref{eq3a} it follows that 
	\begin{equation}\label{eqproof19}
		|X_l| =|z_1+\cdots + z_r| \leq (r-1)|z_r|^{\frac{1}{1+\epsilon}}+|z_r| \leq r |z_r|,
	\end{equation} which implies
	\begin{equation}\label{eqproof22}
		|z_r| \geq \dfrac{|X_l|}{r}.
	\end{equation}	
	If $|X_l| \leq \dfrac{\gamma^{(n_{rs}/{\log \gamma})}}{4}$, then 
	$$\frac{1}{2}\gamma^l<|X_l|=\frac{1}{2}\gamma^l+\frac{1}{2}\eta^l \leq \dfrac{\gamma^{(n_{rs}/{\log \gamma})}}{4}<\dfrac{\gamma^l}{4}$$
	i.e.,
	$$\frac{1}{2}\gamma^l<\dfrac{\gamma^l}{4},$$
	which is not true. So assume that, 
	\begin{equation}\label{eqproof23}
		|X_l|>\dfrac{\gamma^{(n_{rs}/{\log \gamma})}}{4}.
	\end{equation} 
	We rewrite (\ref{eq3a}) using (\ref{eq7a}) as 
	\begin{equation}
		\dfrac{1}{2} \gamma^l-z_r=z_1+\cdots +z_{r-1}-\dfrac{1}{2} \eta^l.
	\end{equation}
	Dividing throughout by $z_r$,
	\begin{equation} \label{eq043}
		\left|2^{-1}\gamma^l z_r^{-1}-1\right| \leq \left|\dfrac{\sum_{i=1}^{r-1}z_i}{z_r}\right| +\left|\dfrac{\eta^l}{2z_r} \right|. 
	\end{equation}
	Next we will find bounds for each term in the right hand side of (\ref{eq043}). \\From (\ref{eqproof22}) and (\ref{eqproof23}),
	\begin{align*}
		\dfrac{\sum_{i=1}^{r-1}|z_i|}{|z_r|}&\leq  (r-1)\dfrac{|z_r|^{\frac{1}{1+\epsilon }}}{|z_r|} \leq (r-1)\dfrac{1}{|z_r|^{\frac{\epsilon}{1+\epsilon }}} \\& \leq (r-1)\dfrac{(4r)^{\frac{\epsilon}{1+\epsilon }}}{\gamma^{\frac{\left({n_{rs}/\log \gamma} \right)\epsilon}{1+\epsilon }}} \leq c_1(s) \dfrac{1}{\gamma^{\frac{\left({n_{rs}/\log \gamma} \right)\epsilon}{1+\epsilon }}},
	\end{align*} where $c_1(s)=(r-1)(4r)^{\frac{\epsilon}{1+\epsilon }}$. To find a bound for $\left|\dfrac{\eta^l}{2z_r} \right| $, 
	$$\left|\dfrac{\eta^l}{2z_r} \right| <\frac{1}{2}\cdot \frac{4r\eta^l}{\gamma^{(n_{rs}/{\log \gamma})}}< 2r\cdot \frac{\eta^l}{\gamma^{(n_{rs}/{\log \gamma})}}<2r \cdot \left(\frac{\eta}{\gamma}\right)^{n_{rs}/{\log \gamma}}.$$
	Substituting these bounds in (\ref{eq043}),
	\begin{align*}
		\left|2^{-1}\gamma^l z_r^{-1}-1\right| &\leq c_1(s) \dfrac{1}{\gamma^{\frac{\left({n_{rs}/\log \gamma} \right)\epsilon}{1+\epsilon }}}+2r \cdot \left(\frac{\eta}{\gamma}\right)^{n_{rs}/{\log \gamma}}\\& \leq c_2(s) \cdot \max \left\{{\dfrac{1}{\gamma^{\frac{\epsilon}{1+\epsilon }}},\frac{\eta}{\gamma}}\right\}^{n_{rs}/{\log \gamma}}.\numberthis{\label{eq044}}
	\end{align*}
	We know that, if $|x-1|<\frac{1}{2}, \text{ then } |\log x|<2|x-1|$. Set
	\[\Lambda:=2^{-1}\gamma^l z_r^{-1}-1 \quad \mbox{and}\quad \Gamma :=-\log 2+l\log \gamma-\log z_r.\] Using the fact that $z_r= p_{1}^{n_{r1}}\cdots p_{s}^{n_{rs}}$,
	$$|\Lambda|=\left|\frac{1}{2}\gamma^l p_1^{-n_{r1}}\cdots p_s^{-n_rs}-1 \right|.$$ If  $|\Lambda|=0$, then  $\gamma^l=2z_r$ is an integer, which is false for any $l\geq 1$. Hence $|\Lambda| \neq 0$.
	So applying Matveev's theorem (Lemma \ref{lem12}), with
	$$(\eta_{1},d_1)=(2,-1),~~~ (\eta_{2},d_2)=(\gamma,l),~~~\text{ and } (\eta_{i+2},d_{i+2})=(p_i,-n_{ri})\text{ with $i=1,\ldots, s.$}$$
	Furthermore,
	$$d_{\mathbb{L}}=2, \; h(2)=\log 2, \quad h(\gamma)=\frac{1}{2}\log \gamma, \quad h(p_j)=\log p_j, \text{where}\; j=1,\ldots ,s.$$
	We choose $A_i$'s as follows,
	$$A_1=0.16, \qquad A_2=\log\gamma, \qquad A_{i+2}=2 \log p_i, \text{with}\; i=1,\ldots, s.$$
	Since $p_s$ is odd, without loss of generality we assume that $\min\{\gamma, n_{rs}\}\geq 2.5 r  p_s>7$. From (\ref{eq108}), $$l<\dfrac{s n_{rs}\log p_s+\log2.5 r }{\log \gamma}.$$  Now if $\min\{\gamma,n_{rs}\}=\gamma$, then $2.5 r p_s<\gamma $, which implies that, $\log p_s< \log \gamma -\log {2.5r}$. Substituting this in above inequality , 
	\begin{align*}
		l&<\dfrac{s n_{rs}(\log \gamma -\log {2.5r})+\log2.5 r }{\log \gamma}<sn_{rs}.
	\end{align*} If $\min\{\gamma,n_{rs}\}= n_{rs}$, then $n_{rs}<\gamma$, so as in the previous case  $l<sn_{rs}$. Hence, in both cases $l<sn_{rs}$. So we can take $B=sn_{rs}$. Applying Matveev's theorem with $t=s+2$, 
	\begin{align*}
		\log|\Lambda|&>-1.4\cdot 30^{s+5}(s+2)^{4.5} 2^2 (1+\log2)(1+\log (s n_{rs}))\cdot \\& \qquad 0.16  (\log \gamma)(2\log p_1)\cdots(2\log p_s)\\ &>-c_3(s)(\log n_{rs}) (\log\gamma).\numberthis\label{eqproof27}
	\end{align*} 
	It is clear that, $\gamma^{\frac{\epsilon}{1+\epsilon}}<\gamma^2$  for $0< \epsilon <1$.
	Now taking logarithm on both sides of \eqref{eq044} and comparing with \eqref{eqproof27},
	$$-c_3(s)(\log n_{rs}) (\log\gamma)<\log c_2(s)-\left(\dfrac{\epsilon n_{rs}}{(1+\epsilon)\log \gamma }\right)\log \gamma,$$ which further gives
	\begin{align*}
		n_{rs}&<\left(\dfrac{1+\epsilon}{\epsilon}\right)\Big(\log c_2(s)+c_3(s)(\log n_{rs}) (\log\gamma)\Big)\\&\leq c_4(s)(\log n_{rs}) (\log\gamma) .\numberthis\label{eq045}
	\end{align*}
	This gives an upper bound for $n_{rs}$ in terms of $\log \gamma$.
	For $i=1,2$ let
	\begin{align*}
		\Gamma_1^{(i)}&:=-\log 2+l_{i}\log \gamma-\log (p_1^{m_{r1}}\cdots p_s^{m_{rs}})\\&=-\log 2+l_i\log \gamma-\sum_{k=1}^{s}m_{rk}\log p_k.
	\end{align*} Note that $m_{rk}=a_{rk} \text{ or } b_{rk} \text{ for $k=1,\ldots, s$}$. Next our aim is to eliminate term involving $\log \gamma$ from $	\Gamma_1^{(i)}$. For that, consider 
	\begin{equation}\label{eq22}
		\Gamma_2:=l_1\Gamma_1^{(2)}-l_2\Gamma_1^{(1)}=(l_2-l_1)\log2+(l_2-l_1)\sum_{k=1}^{s}m_{rk}\log p_k.
	\end{equation}
	Since $l_1<l_2$, then from \eqref{eq044},
	\begin{align}
		\begin{split}\label{eq23}
			|\Gamma_2| &\leq  l_1 |\Gamma_1^{(2)}|+l_2|\Gamma_1^{(1)}|	\\& \leq l_2c_2(s)\cdot \max\left\{ {\dfrac{1}{\gamma^{\frac{\epsilon}{1+\epsilon }}},\frac{\eta}{\gamma}}\right\}^{b_{rs}/{\log \gamma}}+l_2c_2(s)\cdot \max\left\{ {\dfrac{1}{\gamma^{\frac{\epsilon}{1+\epsilon }}},\frac{\eta}{\gamma}}\right\}^{a_{rs}/{\log \gamma}}\\& \leq  l_2c_2(s)\cdot \max\left\{ {\dfrac{1}{\gamma^{\frac{\epsilon}{1+\epsilon }}},\frac{\eta}{\gamma}}\right\}^{a_{rs}/{\log \gamma}}\left(1+\max\left\{ {\dfrac{1}{\gamma^{\frac{\epsilon}{1+\epsilon }}},\frac{\eta}{\gamma}}\right\}^{{b_{rs}/{\log \gamma}-{a_{rs}/{\log \gamma}}}}\right)\\&\leq l_2c_5(s)\cdot \max\left\{ {\dfrac{1}{\gamma^{\frac{\epsilon}{1+\epsilon }}},\frac{\eta}{\gamma}}\right\}^{a_{rs}/{\log \gamma}}. 
		\end{split}
	\end{align} Assume that, $c_5(s)\cdot l_2\max\left\{ {\dfrac{1}{\gamma^{\frac{\epsilon}{1+\epsilon }}},\frac{\eta}{\gamma}}\right\}^{a_{rs}/{\log \gamma}}<1/2$. Otherwise, we can derive a bound which is weaker than the bound with this assumption. Set
	\begin{align*}
		|\Lambda_2|&:=|2^{(l_2-l_1)}p_1^{n_{r1}(l_2-l_1)}\cdots p_s^{n_{rs}(l_2-l_1)}-1|\\&=|e^{\Gamma_2}-1|<2|\Gamma_2|< 2c_5(s)\cdot l_2\max\left\{ {\dfrac{1}{\gamma^{\frac{\epsilon}{1+\epsilon }}},\frac{\eta}{\gamma}}\right\}^{a_{rs}/\log \gamma}\numberthis{\label{eqproof31}}.
	\end{align*}
	We will apply Matveev's theorem for $\Lambda_2$ with
	$$(\eta_{1},d_1):=(2,l_2-l_1), \qquad (\eta_{i+1},d_{i+1}):=(p_i, n_{ri}(l_2-l_1)),\text{ where $i=1,\ldots ,s.$ }$$ Again,
	$d_{\mathbb{L}}=1, \; h(2)=\log 2, \quad  h(p_i)=\log p_i, \text{ with $i=1,\ldots ,s$ }$. Choose 
	$$A_1=2\log2, \qquad A_{i+1}=2\log p_i, \text{ with $i=1,\ldots ,s$ }.$$
	Note that $|n_{ri}(l_2-l_1)| \leq n_{rs}(l_2+l_1)$ and $l_2<s n_{rs}$. Here $B\geq \max\limits_{1\leq i \leq s}\{|l_2-l_1|,|n_{ri}(l_2-l_1)|\}$. So, we can set $B=2s(n_{rs})^2$.
	Applying Matveev's theorem,
	\begin{align*}
		\log|\Lambda_2|&>-1.4\cdot 30^{s+4}\cdot{(s+1)}^{4.5} (1+\log (2s(n_{rs})^2))(2\log 2)\\& \qquad (2\log p_1)\cdots (2\log p_s)\\ &>-c_6(s)\log n_{rs} .
	\end{align*} 
	We have, $\gamma^{\frac{\epsilon}{1+\epsilon}}<\gamma^2$  for $0< \epsilon <1$.  Now comparing the above inequality with \eqref{eqproof31}, we get
	\begin{align*}
		-c_6(s)\log n_{rs}&<\log (2l_2c_5(s))-\frac{a_{rs}}{\log \gamma} \log \max \left\{\gamma^{\frac{\epsilon}{1+\epsilon}}, \gamma^2\right\}\\& \leq \log (2l_2c_5(s))-\frac{a_{rs}}{\log \gamma} \log  \gamma^{\frac{\epsilon}{1+\epsilon}}.
	\end{align*}
	i.e., $$\frac{a_{rs}}{\log \gamma}\cdot \dfrac{\epsilon}{1+\epsilon}\log \gamma < c_6(s)\log n_{rs}+\log (2l_2c_5(s)),$$ which implies that,
	\begin{align*}
		a_{rs}&<\left(\dfrac{1+\epsilon}{\epsilon}\right)\left(\log (2l_2 c_5(s))+c_6(s) \log n_{rs}\right).
		\numberthis \label{eq049}
	\end{align*}
	\newline Considering \eqref{eq045} for $n_{rs}=b_{rs}$,
	$$  b_{rs}<c_4(s)(\log b_{rs}) (\log\gamma).$$
	By Lemma \ref{lemma3.3}, 
	\begin{equation}\label{eq27}
		b_{rs}<2(c_4(s) \log\gamma \cdot \log (c_4(s) \log\gamma)).
	\end{equation}
	Using \eqref{eq108} for $n_{rs}=a_{rs}$ we get,
	$$	a_{rs}\log p_s<l_1 \log \gamma< s a_{rs}\log p_s+\log2.5 r $$ and this implies,
	\begin{equation}\label{eq28}
		\log \gamma<l_1 \log \gamma< s a_{rs}\log p_s+\log2.5 r.
	\end{equation}
	Since $\sqrt{d-1}\leq X_{l_1}\leq rp_s^{sa_{rs}}$, so assume that $\log\gamma< 2a_{rs} s\log p_s$. Then from (\ref{eq27}),
	\begin{equation}\label{eq29}
		b_{rs}<2(c_4(s) 2a_{rs} s\log p_s \cdot \log (c_4(s) 2a_{rs} s\log p_s)).
	\end{equation}
	Since $l_2<sb_{rs}$, so from \eqref{eq049} and \eqref{eq29} we infer,
	\begin{align*}
		a_{rs}&<\left(\dfrac{1+\epsilon}{\epsilon}\right)\Big(\log (2sb_{rs} c_5(s))+c_6(s) \log b_{rs}\Big)
		\\&< \left(\dfrac{1+\epsilon}{\epsilon}\right) \Big( \log (2s2(c_4(s) 2a_{rs} s\log p_s \cdot \log (c_4(s) 2a_{rs} s\log p_s)) c_5(s)) \Big)\\& \qquad+ \left(\dfrac{1+\epsilon}{\epsilon}\right) \Big(c_6(s) \log (2(c_4(s) 2a_{rs} s\log p_s \cdot \log (c_4(s) 2a_{rs} s\log p_s)) )\Big).
	\end{align*}
	By Lemma \ref{lemma3.3} ,
	$$a_{rs}<c_7(s).$$ Thus \eqref{eq28} implies that,
	\begin{equation}\label{eq33}
		\log \gamma <sc_7(s) \log p_s+\log2.5 r,
	\end{equation} where $c_7(s)$ is an effectively computable constant that depends only on $s, p_s \text{ and } \epsilon $.
	So,
	$$\log \gamma <\log (p_s^{sc_{7}(s)}2.5r),$$
	i.e.,
	$$ \gamma <p_s^{sc_{7}(s)}2.5r.$$
	Since $\gamma>2\sqrt{d-2}$, then $$d<p_s^{2sc_{7}(s)}(2.5 r)^2<p_s^{2c_8(s)}r^2.$$
	Also,
	\begin{align*}
		l_2&<b_{rs} s<s2(c_4(s) 2c_{7}(s) s\log p_s \cdot \log (c_4(s) 2c_{7}(s) s\log p_s))=:c_9(s).
	\end{align*}
	Thus, we conclude that there are only finitely many possibilities  for $d$ under the assumption that  $\min\{\gamma, n_{rs}\}\geq 2.5 r  p_s$. If $\min\{\gamma, n_{rs}\} < 2.5 r  p_s$, then again we can obtain bounds on $l, n_{rs}$ and $d$. In particular, if $n_{rs} = \min\{\gamma, n_{rs}\} < 2.5 r  p_s$, then by \eqref{eq108}, we have upper bounds on $l, \log \gamma$ and then on $d$. However, if $\gamma=\min\{\gamma, n_{rs}\} < 2.5 r  p_s$, then there exists an upper bound on $d$ and then on $l$ and $n_{rs}$. This completes the proof of Theorem \ref{th1}.
	\section{Proof of Theorem \ref{th2}}
	Suppose that for $1 \leq l_1 < l_2$, 
	\begin{equation}\label{eq34}
		X_{l_1}=2^{a_1}3^{a_2} \quad \text{ and  }\quad 	X_{l_2}=2^{b_1}3^{b_2},
	\end{equation}
	with $a_1\leq a_2$ and $b_1\leq b_2$. We denote 
	\begin{align*}
		(l, n_{1}, n_{2})&=(l_1, a_{1}, a_{2})
		&\text{  or }&
		&	(l, n_{1}, n_{2})=(l_2, b_{1},b_{2}).
	\end{align*}
	Assume $n_2>n_1$.
	From the proof of Theorem \ref{th1},
	\begin{equation}\label{eq106}
		3^{n_{2}}< X_l=2^{n_1} 3^{n_2}< 3^{2 n_{2}}.
	\end{equation}
	Also 
	\[3^{a_{2}}\leq X_{l_1} \leq X_{l_2}\leq  3^{2 b_{2}}\] and this implies $a_{2}\leq 2 b_{2}$. From (\ref{eq106}) and (\ref{eq31}), we get
	$3^{n_{2}}< \gamma^{l}< 2.5 X_l< 2.5 \cdot 3^{2 n_{2}}$. This gives, 
	\begin{equation}\label{eq18}
		n_{2}\log 3<l \log \gamma < 2 n_{2}\log 3+\log 2.5 .
	\end{equation} 
	Using (\ref{eq7a}), rewrite \eqref{eq3b} as
	$$ \frac{\gamma^l}{ 2^{(n_1+1)}3^{n_2}}-1= \frac{\eta^l}{ 2^{(n_1+1)}3^{n_2}}.$$ Setting $\Lambda:=\gamma^l 2^{-(n_{1}+1)} 3^{-n_2}-1$, we get
	\begin{equation}\label{eq41}
		|\Lambda|	\leq \left|\frac{\eta^l}{ 2^{(n_{1}+1)} 3^{n_2} }\right| \leq \dfrac{1}{3^{n_2}}.
	\end{equation}
	Put 
	\begin{equation}\label{eq41a}
		\Gamma=l \log \gamma-(n_1+1)\log 2-n_2\log 3
	\end{equation} and since $|\Lambda|=|e^{\Gamma}-1|<\dfrac{1}{2}$, so $|\Gamma|<2|\Lambda|<2\cdot 3^{-n_2}$. Since $\gamma $ is not an integer, so $\Lambda \neq 0.$  Hence we can use Matveev's theorem  with $t=3$ and 
	$$(\eta_{1},d_1)=(2,-(n_1+1)), \qquad (\eta_{2},d_2)=(\gamma,l), \qquad (\eta_{3},d_3) =(3,-n_2 ).$$Further, $d_{\mathbb{L}}=2, h(2)=\log 2, h(\gamma)=\frac{1}{2}\log \gamma,  h(3)=\log 3, A_1=2\log 2, A_2=\log\gamma, A_{3}=2 \log 3$. Assume that $\min\{\gamma, n_2\}\geq 7.5$. Then one can see that $l<2n_2$. Applying Matveev's theorem ,
	\begin{align*}
		\log|\Lambda|&>-c_{10}(\log n_{2}) (\log\gamma)\log 3,
	\end{align*} where $c_{10}=1.33 \cdot 10^{14}$.
	Comparing above inequality with (\ref{eq41}), we get 
	\begin{equation}\label{eq022}
		n_{2}<c_{11}(\log n_2)\log \gamma, \text{ with } c_{11}=1.34\cdot 10^{14}.
	\end{equation} 
	To eliminate $\log \gamma$, let
	$$\Gamma_1^{(i)}:=-(m_1+1)\log 2-m_2\log 3+l_i \log \gamma,$$ where $m_i=a_i \text{ or } b_i$ with $i=1, 2$. Define
	\begin{align*}
		\Gamma_2&:=l_1\Gamma_1^{(2)}-l_2\Gamma_1^{(1)}\\&=\left(l_2-l_1+l_2a_1-l_1b_1\right)\log 2+(l_2a_2-l_1b_2)\log 3. \numberthis\label{eq0023}
	\end{align*}
	 Then proceeding as in the proof of Theorem \ref{th1}, we get 
	 \begin{align*}
	 	|\Gamma_2| \leq c_{12}\cdot l_23^{-\frac{a_2}{2}} ~\text{ with $c_{12}=4$,  } 
	 	\numberthis \label{eq43}
	 \end{align*}
	\[b_2<2 \cdot 4c_{11}a_2 \log 3 \cdot \log(4c_{11}a_2 \log 3)< 8.62 \cdot 10^{28} \quad \mbox{and}\quad l_2 < 2b_2< 17.2\cdot 10^{28}.\]
	These bounds are very large for a computational search. Our next aim is to reduce these bounds to smaller numbers using  Lemma \ref{lemma3.4} and \ref{lemma3.5}. From  \eqref{eq0023} and \eqref{eq43}, 
	\begin{equation*}
		\left|\dfrac{\log 2}{\log 3}-\dfrac{(l_1 b_2-l_2 a_2)}{l_2(1+a_1)-l_1(1+b_1)}\right| \leq \dfrac{c_{12}l_2}{3^{\frac{a_2}{2}}\cdot (l_2(1+a_1)-l_1(1+b_1))(\log 3)}
	\end{equation*} 
	To use Lemma \ref{lemma3.4}, put $\tau:=\dfrac{\log 2}{\log 3}$ and $M_1:=8.62\cdot 10^{28}$. With the help of Mathematica finding $q$ such that $q>M_1$ as
	$q_{60}>8.62\cdot 10^{28}$ and $a(M_1)=55$. Now applying Lemma \ref{lemma3.4},
	\begin{align*}
		3^{\frac{a_2}{2}}<\dfrac{c_{12}l_2 (l_2(1+a_1)-l_1(1+b_1))57}{ \log 3}<1824 \cdot M_1^3,
	\end{align*}
	which implies that, 
	\begin{align*}
		a_2< \dfrac{2\log(1824 \cdot M_1^3)}{\log 3}<377.55 
	\end{align*}
	i.e., $a_2 \leq 377$ and this gives $b_2<1.78 \cdot 10^{19}$. To get better bounds repeat this process with new $M$ value. Next put $M_2=1.78 \cdot 10^{19}$. Finding the value of $q$  as, $q_{41}>1.78 \cdot 10^{19}$ and $a(M_2)=55$. Hence $a_2 \leq 255$ and $b_2< 1.2\cdot 10^{19}$. Using Lemma \ref{lemma3.4} again with $M_3= 1.2\cdot 10^{19}$ to get $a_2 \leq 253$. This way will not give any better bounds further. So we will use another method.
	
	Consider the following polynomial with integer coefficients:
	\begin{align*}
		X_l &=\dfrac{1}{2} \left(\gamma^l+\eta^l
		\right)=\dfrac{1}{2}\left((X_1+Y_1 \sqrt{d})^l+(X_l-Y_l \sqrt{d})^l\right)\\&=\dfrac{1}{2}\left(\left(X_1+\sqrt{X_1^2-1}\right)^l+\left(X_1-\sqrt{X_1^2-1}\right)^l\right):=P_l(X_1).
	\end{align*}
	Now from (\ref{eq18}) finding a bound for $l_1$ as,
	\begin{equation}\label{eq034}
		l_1< \dfrac{(2\cdot 253) \log 3+\log 2.5}{\log \gamma}.
	\end{equation} We know that $X_1 \geq \sqrt{d-1}$ and $ \gamma>2\sqrt{d-2}$. Assume that $d>401$, which gives $X_1>20$ and  $l_1 \leq 150$. Now from (\ref{eq34}),
	\begin{equation}\label{eq35}
		P_{l_1}(X_1):=2^{a_1}3^{a_2}
	\end{equation}with $a_2\in [0,253], a_1\in [0,a_2], l_1 \in [2,150]$ and $X_1>20$. Note that $l_1\geq 2$, because if we are taking $l_1=1$ we get $P_{l_1}(X_1)=X_1$. In this case there is nothing to solve.
	
	For the case $l_1>1$ and $d>401$, a quick computer search on (\ref{eq35}) yields nothing; so if $l_1>1$ then $1<d\leq 401$. In this case also, we are not getting any solutions. 
	
	Since  $l_1$ to be minimal, we must have $l_1=1$ for all $d$. Also $\gamma=X_1+\sqrt{X_1^2-1}$. From (\ref{eq41a}),
	$$|l\log \gamma - (n_1+1) \log 2- n_2 \log 3 |<\dfrac{2}{3^{n_2}}.$$ Rearranging and comparing with Lemma \ref{lemma3.5},
	$$ m:=n_2, \tau:=\dfrac{\log 3}{\log \gamma}, \mu:=\dfrac{254\cdot \log 2}{\log \gamma}, A:=\dfrac{2}{\log \gamma}, B:=3,  k:=n_2.$$ We have an upper bound for $b_2$, which is $M_4:= 1.179\cdot 10^{19}$. Reduce this bound again and denote the reduced bound by $M_5(X_1)$ for each $X_1$. Using Mathematica we get,  $\underset{X_1}\max~M_5(X_1)=52$. The $X$-coordinates of (\ref{eqpell}) satisfies the recurrence relation,
	$$X_{l} = 2X_1 X_{l-1}-X_{l-2}$$ and for each $X_1$  a bound for $l_2$ is obtained using \eqref{eq18} denoting it by $l(X_1)$:
	$$l_2<\dfrac{2\cdot M_5(X_1)\log 3+\log2.5}{\log \gamma}:=l(X_1).$$ So, $\underset{X_1}\max~l(X_1)=130$. Now finding the minimum of $\epsilon_1(X_1)$ of Lemma \ref{lemma3.5} as $\underset{X_1}\min~{\epsilon_1}(X_1)=0.0011$. Now it remains to verify whether $\dfrac{X_{l_2}}{3^{b_2}}$ is of the form  $2^{b_1}$. Checking this for each value of $X_1$ in the range
	$$l_2 \in[2,130]  \text{ and }b_2 \in[0,52]$$ using Mathematica, we find no solutions. This completes the proof of the Theorem \ref{th2}.
	
	{\bf Acknowledgment:}  First author's research is supported by UGC Fellowship (Ref No. 221610077314). S.S.R. work is supported by grant from Anusandhan National Research Foundation (ANRF) (File No.:CRG/2022/000268) and from National Board for Higher Mathematics (NBHM), Sanction Order No: 14053.

\end{document}